\documentclass[a4paper]{amsart}
\pdfoutput=1
\usepackage[T1]{fontenc}
\usepackage{lmodern}
\usepackage[utf8]{inputenc}
\usepackage{amssymb}
\usepackage{enumitem}
\setlist[enumerate,1]{label=\textup{(\arabic*)}}

\usepackage{mathrsfs,dsfont,mathtools}
\usepackage{nicefrac}

\usepackage{xcolor}

\usepackage{tikz}
\usetikzlibrary{matrix,arrows,calc}
\tikzset{cd/.style=matrix of math nodes,row sep=2em,column sep=2em, text height=1.5ex, text depth=0.5ex}
\tikzset{cdar/.style=->,auto}
\tikzset{mid/.style={anchor=mid}} 
\tikzset{dar/.style={double,double equal sign distance,-implies}}
\tikzset{narrowfill/.style={inner sep=1pt, fill=white}}
\usepackage{tikz-cd}

\usepackage[pdftitle={Graph morphisms as groupoid actors},
  pdfauthor={Ralf Meyer},
  pdfsubject={Mathematics}
]{hyperref}
\usepackage[lite]{amsrefs}

\renewcommand*{\PrintDOI}[1]{\href{http://dx.doi.org/\detokenize{#1}}{doi: \detokenize{#1}}}

\BibSpec{book}{%
  +{}  {\PrintPrimary}                {transition}
  +{,} { \textit}                     {title}
  +{.} { }                            {part}
  +{:} { \textit}                     {subtitle}
  +{,} { \PrintEdition}               {edition}
  +{}  { \PrintEditorsB}              {editor}
  +{,} { \PrintTranslatorsC}          {translator}
  +{,} { \PrintContributions}         {contribution}
  +{,} { }                            {series}
  +{,} { \voltext}                    {volume}
  +{,} { }                            {publisher}
  +{,} { }                            {organization}
  +{,} { }                            {address}
  +{,} { \PrintDateB}                 {date}
  +{,} { }                            {status}
  +{}  { \parenthesize}               {language}
  +{}  { \PrintTranslation}           {translation}
  +{;} { \PrintReprint}               {reprint}
  +{.} { }                            {note}
  +{.} {}                             {transition}
  +{} { \PrintDOI}                   {doi}
  +{} { available at \url}            {eprint}
  +{}  {\SentenceSpace \PrintReviews} {review}
}

\BibSpec{thesis}{%
  +{}  {\PrintAuthors}                {author}
  +{,} { \textit}                     {title}
  +{:} { \textit}                     {subtitle}
  +{,} { \PrintThesisType}            {type}
  +{,} { }                            {organization}
  +{,} { }                            {address}
  +{,} { \PrintDateB}                 {date}
  +{,} { \eprint}                     {eprint}
  +{,} { }                            {status}
  +{}  { \parenthesize}               {language}
  +{}  { \PrintTranslation}           {translation}
  +{;} { \PrintReprint}               {reprint}
  +{.} { }                            {note}
  +{.} {}                             {transition}
  +{} { \PrintDOI}                   {doi}
  +{} { available at \url}            {eprint}
  +{}  {\SentenceSpace \PrintReviews} {review}
}

\usepackage{microtype}

\numberwithin{equation}{section}
\theoremstyle{plain}
\newtheorem{theorem}[equation]{Theorem}
\newtheorem{lemma}[equation]{Lemma}
\newtheorem{proposition}[equation]{Proposition}
\newtheorem{corollary}[equation]{Corollary}
\theoremstyle{definition}
\newtheorem{definition}[equation]{Definition}

\theoremstyle{remark}
\newtheorem{remark}[equation]{Remark}
\newtheorem{example}[equation]{Example}

\newcommand*{\Qu}{\mathsf{p}}
\newcommand*{\s}{s}
\newcommand*{\rg}{r}

\newcommand*{\C}{\mathbb C}

\newcommand*{\T}{\mathbb T}
\newcommand*{\Mat}{\mathbb M}

\newcommand*{\Mult}{\mathcal M}
\newcommand*{\Cont}{\mathrm C}
\newcommand*{\Contb}{\mathrm{C_b}}
\newcommand*{\reg}{\mathrm{reg}}
\newcommand*{\Slice}[1][U]{\mathcal{#1}}

\newcommand*{\K}{\textup{K}}

\newcommand*{\Cst}{\textup C^*}

\newcommand*{\Star}{\texorpdfstring{$^*$\nb-}{*-}}

\newcommand{\Gr}[1][G]{\mathcal #1}
\newcommand*{\Bisp}[1][X]{\mathcal #1}
\newcommand{\Reg}{\mathcal R}

\newcommand{\Mod}{\mathcal M}
\newcommand{\Bis}{\mathcal{B}}

\newcommand*{\nb}{\nobreakdash}  

\mathtoolsset{mathic}
\DeclarePairedDelimiter{\abs}{\lvert}{\rvert}
\DeclarePairedDelimiterX{\braket}[2]{\langle}{\rangle}{#1\,\delimsize\vert\,\mathopen{}#2}

\DeclarePairedDelimiterX{\setgiven}[2]{\{}{\}}{#1\,{:}\,\mathopen{}#2}

\newcommand*{\defeq}{\mathrel{\vcentcolon=}}

\begin{document}

\title{Graph morphisms as groupoid actors}

\author{Gilles G. de Castro}
\email{gilles.castro@ufsc.br}

\address{Departamento de Matem\'atica, Universidade Federal de Santa Catarina, 88040-970 Florian\'opolis SC, Brazil.}

\author{Ralf Meyer}
\email{rmeyer2@uni-goettingen.de}

\address{Mathematisches Institut\\
  Georg-August Universit\"at G\"ottingen\\
  Bunsenstra\ss{}e 3--5\\
  37073 G\"ottingen\\
  Germany}

\begin{abstract}
  We describe proper actors from the underlying groupoid of a graph
  \(\Cst\)\nb-algebra to another \'etale groupoid in terms of
  bisections.  This allows to understand graph morphisms and the
  \Star{}homomorphisms that they induce more conceptually.  More
  generally, we describe actors from the groupoid model of a groupoid
  correspondence to any \'etale groupoid.  This also covers the
  groupoids associated to self-similar groups and self-similar graphs,
  among others.
\end{abstract}


\thanks{This work is part of the project Graph Algebras partially
  supported by EU grant HORIZON-MSCA-SE-2021 Project 101086394. G. G. de Castro was partially supported by Conselho Nacional de Desenvolvimento Cient\'ifico e Tecnol\'ogico (CNPq) - Brazil, and Funda\c{c}\~ao de Amparo \`a Pesquisa e Inova\c{c}\~ao do Estado de Santa Catarina (FAPESC).}

\maketitle

\section{Introduction}
\label{sec:intro}

Graph \(\Cst\)\nb-algebras are very concrete \(\Cst\)\nb-algebras and
for them we may answer most \(\Cst\)\nb-algebraic questions --~such as
whether they are purely infinite or simple~-- in terms of the
underlying graph.

This article continues the study of ``combinatorial''
\Star{}homomorphisms between graph \(\Cst\)\nb-algebras (see, in
particular, \cite{Castro-DAndrea-Hajac:Relation_morphisms}).  Here we
take a rather different approach and use that graph
\(\Cst\)\nb-algebras are groupoid \(\Cst\)\nb-algebras of \'etale
groupoids.  We may, in fact, describe a nice class of
\Star{}homomorphisms from a graph \(\Cst\)\nb-algebra to the
\(\Cst\)\nb-algebra \(\Cst(\Gr[H])\) of any \'etale
groupoid~\(\Gr[H]\).  Recall that a \emph{bisection} or \emph{slice}
of~\(\Gr[H]\) is an open subset on which both the range map~\(\rg\)
and the source map~\(\s\) of~\(\Gr[H]\) are injective.  The characteristic
function of a compact-open bisection \(\Slice\subseteq \Gr[H]\) is a
partial isometry in \(\Cst(\Gr[H])\), and its range and source
projections are the characteristic functions of \(\rg(\Slice)\) and
\(\s(\Slice)\), respectively.

Now let \(\Gamma = (\rg,\s\colon E\rightrightarrows V)\) be a directed
graph.

\begin{definition}
  \label{def:dCK}
  A \emph{dynamical Cuntz--Krieger \(\Gamma\)\nb-family in a
    groupoid~\(\Gr[H]\)} is a family of compact-open subsets
  \(\Omega_v\subseteq \Gr[H]^0\) for \(v\in V\) and bisections
  \(T_e\subseteq \Gr[H]\) for \(e\in E\) such that
  \begin{enumerate}[label=\textup{(\ref*{def:dCK}.\arabic*)},
    leftmargin=*,labelindent=0em]
  \item \label{en:dCK_1}%
    \(\Omega_v \cap \Omega_w = \emptyset\) for \(v,w\in V\) with
    \(v\neq w\);
  \item \label{en:dCK_2}%
    for all \(e\in E\), \(\s(T_e) = \Omega_{\s(e)}\) and
    \(\rg(T_e) \subseteq \Omega_{\rg(e)}\);
  \item \label{en:dCK_3}%
    \(\rg(T_e) \cap \rg(T_f) = \emptyset\) for \(e,f\in E\) with
    \(e\neq f\);
  \item \label{en:dCK_4}%
    \(\bigsqcup_{\rg(e) = v} \rg(T_e) = \Omega_v\) for each regular
    vertex \(v\in V_\reg\).
  \end{enumerate}
\end{definition}

Given a dynamical Cuntz--Krieger family in~\(\Gr[H]\), their
characteristic functions form a Cuntz--Krieger family in
\(\Cst(\Gr[H])\) as in~\cite{Raeburn:Graph_algebras}.  The latter
induces a nondegenerate \Star{}homomorphism
\(\Cst(\Gamma) \to \Cst(\Gr[H])\), which maps the generators \(p_v\)
and~\(t_e\) of the graph \(\Cst\)\nb-algebra to the characteristic
functions of \(\Omega_v\) and \(T_e\), respectively.  A dynamical
Cuntz--Krieger family is \emph{nondegenerate} if
\(\Gr[H]^0 = \bigsqcup_{v\in V} \Omega_v\).  Most theoretical results
assume this condition.  It is no big loss of generality to restrict
attention to the nondegenerate case because a dynamical Cuntz--Krieger
family in~\(\Gr[H]\) is a nondegenerate dynamical Cuntz--Krieger
family in \(\Gr[H]|_{\Omega}\) for the open subset
\(\Omega = \bigsqcup_{v\in V} \Omega_v\).  The groupoid
\(\Cst\)\nb-algebra of \(\Gr[H]|_{\Omega}\) is a hereditary subalgebra
of that of~\(\Gr[H]\).

Any graph \(\Cst\)\nb-algebra is a groupoid \(\Cst\)\nb-algebra
\(\Cst(\Gamma) \cong \Cst(\Gr_\Gamma)\) for an explicit
groupoid~\(\Gr_\Gamma\), which was described in complete generality by
Paterson~\cite{Paterson:Graph_groupoids}, following the description in
special cases in the seminal
article~\cite{Kumjian-Pask-Raeburn-Renault:Graphs}.  So the
construction above also produces \Star{}homomorphisms between two
graph \(\Cst\)\nb-algebras.  As we shall see, the construction above
covers all the \Star{}homomorphisms between graph \(\Cst\)\nb-algebras
that have been described combinatorially
in~\cite{Castro-DAndrea-Hajac:Relation_morphisms}.  It is far more
general, however.  In particular, all the \Star{}homomorphisms
considered in~\cite{Castro-DAndrea-Hajac:Relation_morphisms} map
vertex projections to sums of vertex projections, but this need not be
true for dynamical Cuntz--Krieger families.

We arrived at the dynamical Cuntz--Krieger families in a different
way, namely, by working out a universal property of the underlying
groupoid~\(\Gr_\Gamma\) of a graph \(\Cst\)\nb-algebra.  This
universal property originally describes actions of~\(\Gr_\Gamma\) on
topological spaces, but it also describes the arrows
from~\(\Gr_\Gamma\) to other \'etale groupoids in a certain category
of \'etale groupoids, which is originally due to Buneci and
Stachura~\cite{Buneci-Stachura:Morphisms_groupoids}.  The arrows in
this category are such that they induce nondegenerate
\Star{}homomorphisms between the groupoid \(\Cst\)\nb-algebras.  We
call these arrows \emph{actors}.  Our main result asserts that there
is a natural bijection between nondegenerate dynamical Cuntz--Krieger
families and actors \(\Gr_\Gamma \to \Gr[H]\).

This categorical interpretation of nondegenerate dynamical
Cuntz--Krieger families has several nice consequences.  First, since
actors are the arrows of a category, they may be composed, and this
carries over to nondegenerate dynamical Cuntz--Krieger families.  In
contrast, the composition
in~\cite{Castro-DAndrea-Hajac:Relation_morphisms} remains rather
mysterious, mostly because of technical extra conditions that are
needed in~\cite{Castro-DAndrea-Hajac:Relation_morphisms}.  Another
important consequence concerns inverses.  We do not know, in general,
whether an actor between two \'etale groupoids that induces an
isomorphism between the groupoid \(\Cst\)\nb-algebras must already be
a groupoid isomorphism.  If, however, the groupoids in question are
Hausdorff and effective, then this is the case.  Using this, we prove
under a mild condition that if an isomorphism of graph
\(\Cst\)\nb-algebras is induced by a dynamical Cuntz--Krieger family,
then so is its inverse.  Example~\ref{exa:concrete_inverse}
illustrates this in a very simple case.

Our approach generalises far beyond graph \(\Cst\)\nb-algebras.  The
same method covers the \(\Cst\)\nb-algebras of self-similar groups and
self-similar graphs.  The theory even has a parallel for topological
graphs or even topological, higher-rank, self-similar graphs.  In this
article, we only consider the rank-one case, however.  Our starting
point is an (\'etale, locally compact) groupoid
correspondence~\(\Bisp\) on an (\'etale, locally compact)
groupoid~\(\Gr\).  This induces a \(\Cst\)\nb-correspondence
\(\Cst(\Bisp)\) on the groupoid \(\Cst\)\nb-algebra \(\Cst(\Gr)\).  In
addition, we choose an open \(\Gr\)\nb-invariant subset
\(\Reg\subseteq \Gr^0\) such that the map
\(\rg_*\colon \Bisp/\Gr \to \Gr^0\) induced by the anchor map of the
left \(\Gr\)\nb-action restricts to a proper map
\(\rg_*^{-1}(\Reg) \to \Reg\).  Then \(\Cst(\Gr|_\Reg)\) is an ideal
in \(\Cst(\Gr)\) that acts on~\(\Cst(\Bisp)\) by compact operators.
So the Cuntz--Pimsner algebra of \(\Cst(\Bisp)\) relative to the ideal
\(\Cst(\Gr|_\Reg)\) is defined.  It is shown in the
dissertation~\cite{Antunes:Thesis} that this relative Cuntz--Pimsner
algebra is the groupoid \(\Cst\)\nb-algebra of a certain groupoid,
called the \emph{groupoid model} of~\(\Bisp\) relative to~\(\Reg\).
This groupoid model is characterised through a universal property that
describes its actions on topological spaces in terms of suitable
``actions'' of \(\Gr\) and~\(\Bisp\).  This allows to describe also
the actors from the groupoid model to another \'etale groupoid.  Our
main reference for this theory
is~\cite{Meyer:Groupoid_models_relative}, where the main results
of~\cite{Antunes:Thesis} are proven and a technical error
in~\cite{Antunes:Thesis} is corrected.

Any graph \(\Cst\)\nb-algebra is a relative Cuntz--Pimsner algebra as
above, where the groupoid~\(\Gr\) is simply a discrete set with only
identity arrows.  So the graph \(\Cst\)\nb-algebra is isomorphic to
the groupoid \(\Cst\)\nb-algebra of an appropriate groupoid model.
The universal property of this groupoid model implies a description of
the actors from the groupoid model to any other groupoid, and these
turn out to be equivalent to the nondegenerate dynamical
Cuntz--Krieger families as described above.  More generally, our
theory applies to all relatively proper groupoid correspondences.

Let us make this explicit for the case of a self-similar group.  We
describe this through a group~\(\Gr\) and a set~\(\Bisp\) with
commuting actions of~\(\Gr\) on the left and right, such that the
right action is free, proper and cocompact.  (This is called a
\emph{permutational bimodule} in~\cite{Nekrashevych:Self-similar}, and
it is one way to describe a self-similarity of~\(\Gr\).)  The groupoid
model for this groupoid correspondence is the groupoid built by
Nekrashevych~\cite{Nekrashevych:Cstar_selfsimilar} in order to attach
a \(\Cst\)\nb-algebra to a self-similar group.  A groupoid actor from
the Nekrashevych groupoid of the self-similarity~\(\Bisp\) of~\(\Gr\)
to another groupoid~\(\Gr[H]\) associates to each \(g\in \Gr\) and
\(x\in \Bisp\) bisections \(T_g\) and~\(T_x\) of~\(\Gr[H]\), such that
the following conditions hold:
\begin{enumerate}
\item for the unit element \(e\in \Gr\), \(T_e\) is the unit
  bisection~\(\Gr[H]^0\);
\item \(T_g T_h = T_{g h}\) if \(g,h\in \Gr \bigsqcup \Bisp\) and at
  most one of them belongs to~\(\Bisp\);
\item \(T_x^* T_x\) is the identity bisection for all \(x\in \Bisp\),
  and \(T_y^* T_x = \emptyset\) if \(x\cdot \Gr \neq y\cdot \Gr\);
\item the ranges of~\(T_x\) for \(x\in \Bisp\) cover~\(\Gr[H]^0\).
\end{enumerate}
This follows from
Corollary~\ref{cor:actor_from_discrete_groupoid_model_bisections}.  By
the second condition above, it is enough to specify the
bisections~\(T_x\) for~\(x\) in a fundamental domain for the free
right \(\Gr\)\nb-action on~\(\Bisp\).  This fundamental domain is the
alphabet used to describe the self-similarity, and the compatibility
between the generators \(T_g\) and~\(T_x\) becomes
\(T_g T_x = T_{g\cdot x} T_{g|_x}\) in the usual notation for
self-similar groups.

\section{Groupoid models of groupoid correspondences}
\label{sec:groupoid_correspondences_models}

In this section, we describe the class of groupoids and
\(\Cst\)\nb-algebras to which our theory applies.  It goes
back to Albandik~\cite{Albandik:Thesis} and is further worked out
in~\cite{Meyer:Diagrams_models}.  We follow
\cites{Antunes-Ko-Meyer:Groupoid_correspondences,
  Meyer:Groupoid_models_relative}, which contain the results of the
dissertation by Antunes~\cite{Antunes:Thesis} most relevant to this
article.  The advantage of the setting
in~\cite{Meyer:Groupoid_models_relative} over that
in~\cite{Albandik:Thesis} is that it also covers relative
Cuntz--Pimsner algebras and thus the \(\Cst\)\nb-algebras of irregular
graphs.

Throughout this article, groupoids are tacitly assumed to be \'etale
and locally compact with Hausdorff object space.  We allow
non-Hausdorff groupoids because they arise automatically from the
theory.  For instance, the Nekrashevych groupoid of a self-similar
group may fail to be Hausdorff, even in important examples like the
Grigorchuk group
(compare~\cite{Clark-Exel-Pardo-Sims-Starling:Simplicity_non-Hausdorff}).

We define actions of groupoids on topological spaces in the standard
way (see
\cite{Antunes-Ko-Meyer:Groupoid_correspondences}*{Definition~2.3}).
We always denote the anchor maps of left actions by~\(\rg\) and those
of right actions by~\(\s\).  So all kinds of products \(x\cdot y\) are
defined if \(\s(x) = \rg(y)\).

A (locally compact, \'etale) \emph{groupoid correspondence} between
two groupoids \(\Gr\) and~\(\Gr[H]\) is a topological space~\(\Bisp\)
with commuting actions of~\(\Gr\) on the left and~\(\Gr[H]\) on the
right, such that the right action is free and proper and the right
anchor map \(\s\colon \Bisp\to \Gr[H]^0\) is \'etale, that is, a local
homeomorphism (see
\cite{Antunes-Ko-Meyer:Groupoid_correspondences}*{Definition~3.1});
this forces~\(\Bisp\) to be locally Hausdorff and locally compact.  A
groupoid correspondence is \emph{proper} if the map
\(\rg_*\colon \Bisp/\Gr[H] \to \Gr^0\) induced by the left anchor map
\(\rg\colon \Bisp\to \Gr^0\) is proper, and \emph{regular}
if~\(\rg_*\) is proper and surjective.  As
in~\cite{Antunes-Ko-Meyer:Groupoid_correspondences}, we write a
groupoid correspondence as an arrow
\(\Bisp\colon \Gr\leftarrow \Gr[H]\).  Groupoid correspondences may be
composed and form what is called a bicategory.  We shall, however, not
use this extra structure explicitly here.

A groupoid correspondence \(\Bisp\colon \Gr\leftarrow \Gr[H]\) induces
a \(\Cst\)\nb-correspondence
\[
  \Cst(\Bisp)\colon \Cst(\Gr) \leftarrow \Cst(\Gr[H]),
\]
that is, \(\Cst(\Bisp)\) is a right Hilbert \(\Cst(\Gr[H])\)-module
with a nondegenerate left action of \(\Cst(\Gr)\) by adjointable
operators (see
\cite{Antunes-Ko-Meyer:Groupoid_correspondences}*{Section~7}).  This
left action is by compact operators --~making \(\Cst(\Bisp)\) a
\emph{proper} correspondence~-- if and only if the groupoid
correspondence~\(\Bisp\) is proper (see
\cite{Antunes-Ko-Meyer:Groupoid_correspondences}*{Theorem~7.14}).  So
for a proper correspondence, we may form the \emph{absolute}
Cuntz--Pimsner algebra \(\mathcal{O}_{\Cst(\Bisp)}\), which is defined
by imposing the Cuntz--Pimsner covariance condition on all of
\(\Cst(\Gr)\), even if the latter does not act faithfully
on~\(\Cst(\Bisp)\).

\begin{definition}[\cite{Meyer:Groupoid_models_relative}]
  Let \(\Gr\) and~\(\Gr[H]\) be groupoids and let
  \(\Reg\subseteq \Gr\) be an open \(\Gr\)\nb-invariant subset.  A
  groupoid correspondence \(\Bisp\colon \Gr \leftarrow \Gr[H]\) is
  called \emph{proper on~\(\Reg\)} if the map from
  \(\rg_{\Bisp}^{-1}(\Reg)/\Gr[H] \subseteq \Bisp/\Gr[H]\) to~\(\Reg\)
  induced by \(\rg_{\Bisp}\colon \Bisp\to \Gr^0\) is proper.
\end{definition}

Let~\(\Bisp\) be a groupoid correspondence that is regular on an open
\(\Gr\)\nb-invariant subset \(\Reg\subseteq \Gr^0\).  Then
\(\Cst(\Gr_\Reg)\) is an ideal in \(\Cst(\Gr)\).  This ideal acts by
compact operators on \(\Cst(\Bisp)\) by
\cite{Antunes-Ko-Meyer:Groupoid_correspondences}*{Corollary~7.14}.  In
particular, we may take \(\Reg= \emptyset\).  Then the Cuntz--Pimsner
algebra of \(\Cst(\Bisp)\) relative to \(\Cst(\Gr_\Reg)\) is just the
Toeplitz \(\Cst\)\nb-algebra of \(\Cst(\Bisp)\).

\begin{example}[\cite{Antunes-Ko-Meyer:Groupoid_correspondences}*{Example~4.1}]
  \label{exa:graph_as_corr}
  Let \(\Gr=V\) and \(\Gr[H]=W\) be locally compact spaces turned into
  \'etale groupoids with only identity arrows.  An action of~\(V\) on
  a locally Hausdorff, locally compact space~\(E\) is just the anchor
  map \(E\to V\).  The resulting \(V\)\nb-action is free and proper if
  and only if~\(E\) is Hausdorff.  Thus a groupoid correspondence
  \(E\colon V\leftarrow W\) is the same as a Hausdorff, locally
  compact space~\(E\) with a continuous map \(\rg\colon E\to V\) and a
  local homeomorphism \(\s\colon E\to W\).  This is the same as a
  topological correspondence (see
  \cite{Katsura:class_I}*{Definition~1.2}).  If \(V=W\), this is
  called a topological graph.  The Cuntz--Pimsner algebra of
  \(\Cst(\Bisp)\) relative to the open subset~\(\Reg\) of regular
  vertices is the topological graph \(\Cst\)\nb-algebra as defined by
  Katsura~\cite{Katsura:class_I}.  If, in addition, \(V\) is discrete,
  then so is~\(E\), and then the groupoid correspondence becomes an
  ordinary directed graph.  A topological graph is proper as a
  groupoid correspondence if and only if \(\rg\colon E\to V\) is
  proper.  In the discrete case, this amounts to the graph being
  row-finite.
\end{example}

\begin{example}[\cite{Antunes-Ko-Meyer:Groupoid_correspondences}*{Example~4.2}]
  \label{exa:self-similar_group_corr}
  Let~\(\Gr\) be a discrete group.  A proper groupoid correspondence
  \(\Bisp\colon \Gr\leftarrow \Gr\) is a discrete set~\(\Bisp\) with
  commuting actions of~\(\Gr\) on the left and right, such that the
  right action is free, proper and cocompact.  This is almost what is
  called a \emph{permutational bimodule}
  in~\cite{Nekrashevych:Self-similar}, except that we do not require
  the induced action on the space of finite words to be faithful.  The
  permutational bimodule is one way to describe a self-similarity
  of~\(\Gr\).
\end{example}

The main result of~\cite{Meyer:Groupoid_models_relative} says that the
relative Cuntz--Pimsner algebra
\(\mathcal{O}_{\Cst(\Bisp),\Cst(\Gr_\Reg)}\) is a groupoid
\(\Cst\)\nb-algebra of a certain groupoid naturally associated to the
data.  We now describe this groupoid by describing its actions.
Let~\(\Gr\) be a groupoid, let \(\Reg\subseteq \Gr^0\) be open and
\(\Gr\)\nb-invariant, and let \(\Bisp\colon \Gr\leftarrow \Gr\) be a
groupoid correspondence that is regular on~\(\Reg\).

\begin{definition}[compare \cite{Meyer:Groupoid_models_relative}*{Definition~3.1}]
  \label{def:correspondence_action}
  Let~\(Y\) be a topological space.  An \emph{action}
  of~\((\Gr,\Bisp,\Reg)\) on~\(Y\) consists of a groupoid action
  of~\(\Gr\) on~\(Y\) and a continuous, open map~\(\mu\) from the
  fibre product \(\Bisp\times_{\s,\Gr^0,\rg} Y\) to~\(Y\), denoted
  multiplicatively as \((x,y)\mapsto x\cdot y\), such that
  \begin{enumerate}[label=\textup{(\ref*{def:correspondence_action}.\arabic*)},
    leftmargin=*,labelindent=0em]
  \item \label{en:diagram_dynamical_system_1}%
    \(\rg(x\cdot y) = \rg(x)\) and
    \(g\cdot (x\cdot y) = (g \cdot x)\cdot y\) for \(g\in \Gr\),
    \(x\in \Bisp\), \(y\in Y\) with \(\s(g) = \rg(x)\) and
    \(\s(x) = \rg(y)\);
  \item \label{en:diagram_dynamical_system_2}%
    if \(x,x'\in \Bisp\), \(y,y'\in Y\) satisfy \(\s(x) = \rg(y)\) and
    \(\s(x') = \rg(y')\), then \(x\cdot y = x'\cdot y'\) if and only
    if there is \(g\in \Gr\) with \(x' = x\cdot g\) and
    \(y = g\cdot y'\);
  \item \label{en:diagram_dynamical_system_3}%
    \(\rg^{-1}(\Reg) \subseteq Y\) is contained in the image
    of~\(\mu\);
  \item \label{en:diagram_dynamical_system_4}%
    if \(K\subseteq \Bisp/\Gr\) is compact, then the set of all
    \(x\cdot y\) with \(x\in \Bisp\), \(y\in Y\) with \([x]\in K\),
    \(\s(x) = \rg(y)\) is closed in~\(Y\).
  \end{enumerate}
  A map \(\varphi\colon Y_1 \to Y_2\) between two \(\Bisp\)
  \nb-actions is \emph{\((\Gr,\Bisp)\)-equivariant} if
  \(\varphi(g\cdot y) = g\cdot \varphi(y)\) for all \(g\in \Gr\),
  \(y\in Y\) with \(\s(g) = \rg(y)\),
  \(\varphi(x\cdot y) = x\cdot \varphi(y)\) for all \(x\in \Bisp\),
  \(y\in Y\) with \(\s(x) = \rg(y)\), and
  \(\varphi^{-1}(\Bisp\cdot Y_2) = \Bisp\cdot Y_1\).
\end{definition}

\begin{definition}
  \label{def:groupoid_model}
  A \emph{groupoid model} for \(\Bisp\) relative to~\(\Reg\) is an
  \'etale groupoid~\(\Mod\) such that for all topological
  spaces~\(Y\), there is a natural bijection between
  \((\Gr,\Bisp,\Reg)\)\nb-actions and \(\Mod\)\nb-actions on~\(Y\).
  Here naturality means that a continuous map between two spaces is
  \((\Gr,\Bisp)\)-equivariant if and only if it is
  \(\Mod\)\nb-equivariant for the associated \(\Mod\)\nb-actions.
\end{definition}

\begin{theorem}
  \label{the:CP_is_groupoid_Cstar}
  A groupoid model~\(\Mod\) exists and is unique up to canonical
  isomorphism.  Its groupoid \(\Cst\)\nb-algebra is naturally
  isomorphic to the Cuntz--Pimsner algebra of \(\Cst(\Bisp)\) relative
  to the ideal \(\Cst(\Gr_\Reg)\) in \(\Cst(\Gr)\).
\end{theorem}

\begin{proof}
  A particular groupoid model is constructed
  in~\cite{Meyer:Groupoid_models_relative}, and
  \cite{Meyer:Groupoid_models_relative}*{Proposition~3.12} asserts
  that any two groupoid models are canonically isomorphic.  The
  isomorphism between its groupoid \(\Cst\)\nb-algebra and the
  relative Cuntz--Pimsner algebra is
  \cite{Meyer:Groupoid_models_relative}*{Theorem~8.8}.
\end{proof}

\begin{remark}
  In the absolute case \(\Reg=\Gr^0\),
  Theorem~\ref{the:CP_is_groupoid_Cstar} goes back to
  Albandik~\cite{Albandik:Thesis}.  He also covers diagrams of proper
  groupoid correspondences.  These correspond to regular higher-rank
  (topological, self-similar) graphs.  The results below carry over to
  diagrams of correspondences in a straightforward way.
\end{remark}

We may describe \((\Gr,\Bisp,\Reg)\)-actions in a different way using
bisections in \(\Gr\) and~\(\Bisp\).  This will also be useful in this
article.  An open subset~\(\Slice\) in a groupoid
correspondence~\(\Bisp\) is a \emph{bisection} if
\(\s\colon \Bisp\to\Gr^0\) and \(\Qu\colon \Bisp\to \Bisp/\Gr\) are
injective on~\(\Slice\).  For the identity correspondence~\(\Gr\),
this specialises to the usual concept of a bisection of a groupoid,
that is, \(\s\) and~\(\rg\) being injective on an open
subset~\(\Slice\).  Let \(\Bis(\Bisp)\) denote the set of bisections
in~\(\Bisp\) and let \(\Bis\defeq \Bis(\Gr) \sqcup \Bis(\Bisp)\).  For
a space~\(Y\), let \(I(Y)\) denote the inverse semigroup of all
partial homeomorphisms on~\(Y\).

\begin{lemma}[\cite{Meyer:Diagrams_models}*{Lemma~4.4}]
  \label{lem:action_from_theta}
  Let~\(Y\) be a space and let \(\vartheta\colon \Bis\to I(Y)\) be a
  map.  It comes from a \((\Gr,\Bisp,\Reg)\)\nb-action on~\(Y\) if and
  only if
  \begin{enumerate}[label=\textup{(\ref*{lem:action_from_theta}.\arabic*)},
    leftmargin=*,labelindent=0em]
  \item \label{en:action_from_theta1}%
    \(\vartheta(\Slice \Slice[V]) = \vartheta(\Slice)\vartheta(\Slice[V])\)
    for all \(\Slice,\Slice[V]\in \Bis\);
  \item \label{en:action_from_theta2}%
    \(\vartheta(\Slice_1)^*\vartheta(\Slice_2) =
    \vartheta(\braket{\Slice_1}{\Slice_2})\) for all
    \(\Slice_1,\Slice_2\in\Bis(\Bisp)\);
  \item \label{en:action_from_theta4}%
    \(\vartheta(\Gr^0)\) is the identity map on all of~\(Y\) and
    \(\vartheta(\emptyset)\) is the empty partial map;
  \item \label{en:action_from_theta5}%
    the restriction of~\(\vartheta\) to open subsets of~\(\Gr^0\)
    commutes with arbitrary unions;
  \item \label{en:action_from_theta3}%
    the images of~\(\vartheta(\Slice)\) for \(\Slice\in\Bis(\Bisp)\) cover the
    domain of~\(\vartheta(\Reg)\);
  \item \label{en:action_from_theta6}%
    if \(\Slice\in\Bis(\Bisp)\) is precompact in~\(\Bisp\), then the
    closure of the codomain of \(\vartheta(\Slice)\) is contained in the
    union of the codomains of~\(\vartheta(W)\) for \(W\in\Bis(\Bisp)\).
  \end{enumerate}
  The corresponding \((\Gr,\Bisp,\Reg)\)-action on~\(Y\) is unique if it
  exists.
\end{lemma}


\section{Groupoid actors}
\label{sec:groupoid_actors}

Next we describe the categories of groupoid actors and proper groupoid
actors; actors are also called \emph{algebraic morphisms} or
\emph{Zakrzewski morphisms}.  These categories are due to Buneci and
Stachura~\cite{Buneci-Stachura:Morphisms_groupoids}, also in the
generality of locally compact groupoids with Haar system, where some
extra difficulties with the measures arise.  We shall only consider
the easy case of \'etale groupoids here.

\begin{definition}
  \label{def:groupoid_actor}
  Let \(\Gr\) and~\(\Gr[H]\) be \'etale groupoids.  An \emph{actor}
  \(\Gr\to \Gr[H]\) is an action of~\(\Gr\) on the arrow space
  of~\(\Gr[H]\) that commutes with the right translation action
  of~\(\Gr[H]\).  The actor is called \emph{proper} if the map
  \(\Gr[H]/\Gr[H] \cong \Gr[H]^0\to \Gr^0 \) induced by the anchor map
  \(\Gr[H]\to \Gr^0\) of the action is proper.
\end{definition}

One reason why groupoid actors are important is that they are
equivalent to certain functors between the action categories of the
groupoids:

\begin{proposition}
  \label{pro:actors_and_actions}
  Let \(\Gr\) and~\(\Gr[H]\) be two \'etale groupoids.  There is a natural
  bijection between actors \(\Gr\to \Gr[H]\) and functors from the category
  of \(\Gr[H]\)\nb-actions on topological spaces to the category of
  \(\Gr\)\nb-actions on topological spaces that preserve the underlying
  space.
\end{proposition}

\begin{proof}
  This is mostly contained in
  \cite{Meyer-Zhu:Groupoids}*{Proposition~4.23}.  The extra fact that
  the functors in question automatically preserve invariant maps is
  shown in \cite{Meyer:Diagrams_models}*{Lemma~4.12}.
\end{proof}

\begin{corollary}
  \label{cor:actors_category}
  Groupoid actors form a category in a canonical way.
\end{corollary}

\begin{proof}
  The composite of two functors between the groupoid action categories
  that preserve the underlying spaces is again a functor of this form.
  And identity functors also have this extra property.
\end{proof}

Explicitly, the composite of two actors \(\Gr\to \Gr[H]\) and
\(\Gr[H]\to \Gr[K]\) is given by the unique action of~\(\Gr\) on the arrow
space of~\(\Gr[K]\) that satisfies
\((g\cdot h)\cdot k = g\cdot (h\cdot k)\) when \(g\in \Gr\),
\(h\in \Gr[H]\) and \(k\in \Gr[K]\) are composable.  The unit arrow
on~\(\Gr\) is simply the left translation action of~\(\Gr\) on its own
arrow space.

\begin{remark}
  \label{rem:actor_action_on_objects}
  Any groupoid~\(\Gr[H]\) acts on its own object space~\(\Gr[H]^0\)
  and on its arrow space~\(\Gr[H]\) by left translation, and the range
  map \(\Gr[H]\to \Gr[H]^0\) is \(\Gr[H]\)\nb-equivariant.  By
  Proposition~\ref{pro:actors_and_actions}, an actor \(\Gr\to \Gr[H]\)
  gives rise to actions of~\(\Gr\) on \(\Gr[H]\) and~\(\Gr[H]^0\) that
  make the range map \(\Gr[H]\to \Gr[H]^0\) equivariant.  In fact, the
  \(\Gr\)\nb-action on~\(\Gr[H]\) is the one that we used above to
  define the actor.  The action on~\(\Gr[H]^0\) has an anchor map
  \(\varrho\colon \Gr[H]^0 \to \Gr^0\).  Since
  \(\rg\colon \Gr[H]\to \Gr[H]^0\) is equivariant, the anchor map
  \(\Gr[H]\to \Gr^0\) must be the composite \(\varrho\circ \rg\) and
  the \(\Gr\)\nb-action on~\(\Gr[H]^0\) must satisfy
  \(g\cdot \omega = \rg(g\cdot 1_\omega)\) for all \(g\in \Gr\),
  \(\omega \in \Gr[H]^0\) with \(\s(g) = \varrho(\omega)\) because
  \(\rg(1_\omega) = \omega\).  So this formula defines the
  \(\Gr\)\nb-action on~\(\Gr[H]^0\).
\end{remark}

Actors and their composition are studied in~\cite{Meyer-Zhu:Groupoids}
in greater generality, for groupoids in a category with pretopology.
This theory treats \'etale groupoids, Lie groupoids, and other
categories of groupoids in a uniform way, but it fails to take into
account Haar systems.

\begin{proposition}
  \label{actor_to_Cstar}
  An actor \(\Gr\to \Gr[H]\) induces a nondegenerate
  \Star{}homomorphism from \(\Cst(\Gr)\) to the multiplier algebra of
  \(\Cst(\Gr[H])\).  This takes values in \(\Cst(\Gr[H])\) if and only
  if the actor is proper.  The \Star{}homomorphism maps the subalgebra
  \(\Cont_0(\Gr^0) \subseteq \Cst(\Gr)\) to \(\Contb(\Gr[H]^0)\), and
  to \(\Cont_0(\Gr[H]^0)\) in the proper case.
\end{proposition}

\begin{proof}
  This follows readily from the results
  of~\cite{Antunes-Ko-Meyer:Groupoid_correspondences}, even for
  non-Hausdorff groupoids.

  The right translation action of an \'etale groupoid on its arrow
  space is always free and proper and has an \'etale anchor map.  So a
  left action of~\(\Gr\) on \(\Bisp=\Gr[H]\) that commutes with the
  right translation action of~\(\Gr[H]\) gives rise to a groupoid
  correspondence.  The associated Hilbert \(\Cst(\Gr[H])\)-module
  \(\Cst(\Bisp)\) is simply the \(\Cst\)\nb-algebra \(\Cst(\Gr[H])\)
  with its canonical Hilbert module structure.  The adjointable and
  compact operators on \(\Cst(\Gr[H])\) are naturally isomorphic to
  \(\Mult(\Cst(\Gr[H]))\) and \(\Cst(\Gr[H])\), respectively.  Thus
  the left action of \(\Cst(\Gr)\) on \(\Cst(\Bisp)\) for a groupoid
  correspondence specialises to a nondegenerate \Star{}homomorphism
  from \(\Cst(\Gr)\) to \(\Mult(\Cst(\Gr[H]))\).  By
  \cite{Antunes-Ko-Meyer:Groupoid_correspondences}*{Theorem~7.14},
  this \Star{}homomorphism takes values in \(\Cst(\Gr[H])\) if and
  only if the left action induces a proper map from
  \(\Bisp/\Gr[H] \cong \Gr[H]^0\) to~\(\Gr^0\), that is, the actor is
  proper.

  If \(f\in\Cont_0(\Gr^0)\), then it acts on \(\Cst(\Gr[H])\) by
  pointwise multiplication with the function \(f\circ \varrho\), where
  \(\varrho\colon\Gr[H]^0 \to \Gr^0\) is the anchor map in
  Remark~\ref{rem:actor_action_on_objects}.  This follows easily from
  the formula for the left action in
  \cite{Antunes-Ko-Meyer:Groupoid_correspondences}*{Equation~(7.3)}.
  Thus the \Star{}homomorphism \(\Cst(\Gr) \to \Mult(\Cst(\Gr[H]))\)
  maps \(\Cont_0(\Gr^0)\) to \(\Contb(\Gr[H]^0)\), and to
  \(\Cont_0(\Gr[H]^0)\) in the proper case.
\end{proof}

Proposition~\ref{actor_to_Cstar} shows that a proper actor between two
groupoids induces a \Star{}homomorphism between the \(\K\)\nb-theory
groups of their groupoid \(\Cst\)\nb-algebras.  If the actor is not
proper, this usually fails.

An actor is invertible in the category of actors if and only if it is
a groupoid isomorphism turned into an actor
(see~\cite{Meyer-Zhu:Groupoids}).  It is easy to see that a groupoid
isomorphism is invertible also as an actor.  In general, we do not
know whether an actor \(\Gr\to\Gr[H]\) must be an isomorphism if it
induces an isomorphism \(\Cst(\Gr) \to \Cst(\Gr[H])\) on the groupoid
\(\Cst\)\nb-algebras.  Using Renault's theory of Cartan inclusions, we
may prove the following special case:

\begin{theorem}
  \label{the:Cartan_isomorphism}
  Let \(\Gr\) and~\(\Gr[H]\) be \'etale, locally compact groupoids and
  assume~\(\Gr\) to be effective and Hausdorff.  If an actor
  \(\Gr\to\Gr[H]\) induces an isomorphism
  \(\Cst(\Gr) \to \Cst(\Gr[H])\), then it is a groupoid isomorphism.
\end{theorem}

\begin{proof}
  The assumptions imply that \(\Cont_0(\Gr^0) \subseteq \Cst(\Gr)\) is
  a Cartan subalgebra (the separability assumption
  in~\cite{Renault:Cartan.Subalgebras} may be left out, compare
  \cite{Kwasniewski-Meyer:Cartan}*{Corollary~7.6}).  In particular,
  \(\Cont_0(\Gr^0)\) is maximal abelian in \(\Cst(\Gr)\).  Since the
  isomorphism \(\Cst(\Gr) \to \Cst(\Gr[H])\) maps \(\Cont_0(\Gr^0)\)
  to \(\Cont_0(\Gr[H]^0)\) and the latter is Abelian, it follows that
  \(\Cont_0(\Gr^0)\) is mapped isomorphically onto
  \(\Cont_0(\Gr[H]^0)\).  Thus
  \(\Cont_0(\Gr[H]^0) \subseteq \Cst(\Gr[H])\) is a Cartan subalgebra
  as well and the actor induces an isomorphism of Cartan subalgebras.
  Since the construction of a twisted groupoid from a Cartan
  subalgebra is functorial, it follows that our actor is a groupoid
  isomorphism.
\end{proof}

\section{Actors out of groupoid models}
\label{sec:actors_mod}

Throughout this section, let \(\Gr\) and~\(\Gr[H]\) be groupoids, let
\(\Reg\subseteq \Gr^0\) be open and \(\Gr\)\nb-invariant, and let
\(\Bisp\colon \Gr\leftarrow \Gr\) be a groupoid correspondence that is
regular on~\(\Reg\).  Let~\(\Mod\) be the groupoid model of
\((\Gr,\Bisp,\Reg)\).  We want to describe the proper groupoid actors
\(\Mod\to \Gr[H]\), using the definition of the groupoid model in
Definition~\ref{def:groupoid_model}.

Recall that given a proper actor \(\Mod\to \Gr[H]\), there is an
induced \(\Mod\)\nb-action on~\(\Gr[H]^0\), the range map
\(\Gr[H]\to \Gr[H]^0\) is \(\Mod\)\nb-equivariant, and the anchor map
\(\Gr[H]^0 \to \Mod^0\) is proper.  This implies the technical
condition~\ref{en:diagram_dynamical_system_4} by
\cite{Meyer:Groupoid_models_relative}*{Lemma~3.7}.  So we may
disregard this condition for proper actors.

By definition, an action of~\(\Mod\) on~\(\Gr[H]\) is equivalent to an
action of~\(\Gr\) on~\(\Gr[H]\) together with a map
\(\Bisp\times_{\s,\Gr^0,\rg} \Gr[H]^0 \to \Gr[H]^0\) subject to some conditions.  We
denote both the \(\Gr\)\nb-action and the map
\(\Bisp\times_{\s,\Gr^0,\rg} \Gr[H]^0 \to \Gr[H]^0\) multiplicatively in the
following.  We first examine when the \(\Mod\)\nb-action induced by
the maps above commutes with the right translation action
of~\(\Gr[H]\):

\begin{lemma}
  \label{lem:actor_commutation}
  The action of~\(\Mod\) on~\(\Gr[H]\) induced by an action of
  \((\Gr,\Bisp,\Reg)\) on~\(\Gr[H]\) commutes with the right
  translation action of~\(\Gr[H]\) if and only if
  \begin{itemize}
  \item the \(\Gr\)\nb-action commutes with the \(\Gr[H]\)\nb-action,
    that is, it is an actor \(\Gr\to \Gr[H]\);
  \item the \(\Bisp\)\nb-action commutes with the
    \(\Gr[H]\)\nb-action, that is, its anchor map
    \(\rg_{\Gr^0}\colon \Gr[H]\to \Gr^0\) factors as
    \(\varrho\circ \rg_{\Gr[H]^0}\) with the range map
    \(\rg_{\Gr[H]^0}\colon \Gr[H]\to \Gr[H]^0\) and some continuous
    map \(\varrho\colon \Gr[H]^0 \to \Gr^0\) and
    \((x\cdot h_1)\cdot h_2 = x\cdot (h_1\cdot h_2)\) for all
    \(x\in \Bisp\), \(h_1,h_2\in \Gr[H]\) with
    \(\s(x) = \rg_{\Gr^0}(h_1)\) and \(\s(h_1) = \rg(h_2)\).
  \end{itemize}
\end{lemma}

\begin{proof}
  If the action of~\(\Mod\) commutes with that of~\(\Gr[H]\), then the
  anchor map \(\s\colon \Gr[H]\to \Gr[H]^0\) of the
  \(\Gr[H]\)\nb-action must be \(\Mod\)\nb-invariant and for any
  bisection~\(V\) in~\(\Gr[H]\), the partial map
  \(\vartheta_V\colon \s^{-1}(\rg(V)) \to \s^{-1}(\s(V))\) of right
  multiplication by~\(V\) is \(\Mod\)\nb-equivariant; this involves
  the restrictions of the \(\Mod\)\nb-action to the open subsets
  \(\s^{-1}(\rg(V))\) and \(\s^{-1}(\s(V))\), which are
  \(\Mod\)\nb-invariant because~\(\s\) is \(\Mod\)\nb-invariant.  Now
  a property of the groupoid model is that a map
  \(f\colon \Gr[H]\to Z\) is \(\Mod\)\nb-invariant if and only if
  \(f(g\cdot h) = f(h)\) and \(f(x\cdot h)= f (h)\) for all
  \(g\in \Gr\), \(h\in \Gr[H]\), \(x\in \Bisp\) with
  \(\s(g) = \s(x) = \rg_{\Gr^0}(h)\).  For the absolute groupoid
  model, this is shown in \cite{Meyer:Diagrams_models}*{Lemma~4.12},
  and the proof is purely formal and therefore carries over to the
  relative case as well.  It is part of the definition of a groupoid
  model that a map between \(\Mod\)\nb-actions is equivariant if and
  only if it is \((\Gr,\Bisp)\)-equivariant.  So the actions of
  \(\Mod\) and~\(\Gr[H]\) commute if and only if
  \(\s\colon \Gr[H]\to \Gr[H]^0\) is invariant and the right
  multiplication by bisections is \((\Gr,\Bisp)\)-equivariant.

  Since~\(\Gr\) is a groupoid, we may argue similarly but backwards
  for the left action of~\(\Gr\).  So \(\s\colon \Gr[H]\to \Gr[H]^0\)
  is \(\Gr\)\nb-invariant and the right multiplication by bisections
  is \(\Gr\)\nb-equivariant if and only if the actions of \(\Gr\)
  and~\(\Gr[H]\) commute.  If the action of~\(\Bisp\) commutes with
  the right multiplication by bisections, this forces the anchor map
  \(\rg_{\Gr^0}\colon \Gr[H]\to \Gr^0\) to factor through the orbit
  space projection of the right multiplication map, which is
  equivalent to \(\rg\colon \Gr[H]\to \Gr[H]^0\).  So we get the
  factorisation of~\(\rg_{\Gr^0}\) through a continuous map
  \(\varrho\colon \Gr[H]^0 \to \Gr^0\).  Recall that a bisection
  \(\Slice\subseteq\Gr[H]\) acts on~\(\Gr[H]\) on the right by the partial
  homeomorphism \(\s^{-1}(\rg(\Slice)) \to \s^{-1}(\s(\Slice))\),
  \(g \mapsto g\cdot h\), for the unique \(h\in\Slice\) with
  \(\rg(h) = \s(g)\).  Therefore, the left multiplication by elements
  of~\(\Bisp\) commutes with~\(\vartheta_{\Slice}\) for all bisections
  \(\Slice\subseteq\Gr[H]\) if and only if
  \((x\cdot h_1)\cdot h_2 = x\cdot (h_1\cdot h_2)\) holds whenever
  \(\s(x) = \rg_{\Gr[H]^0}(h_1)\) and \(\s(h_1) = \rg(h_2)\).  The
  technical condition
  \(\varphi^{-1}(\Bisp\cdot Y_2) = \Bisp\cdot Y_1\) for a
  \((\Gr,\Bisp)\)-equivariant map \(Y_1 \to Y_2\) in
  Definition~\ref{def:correspondence_action} holds automatically here
  because the relevant maps~\(\vartheta_{\Slice}\) are invertible and
  their inverse~\(\vartheta_{\Slice^*}\) is of the same form
  as~\(\vartheta_{\Slice}\).
\end{proof}

\begin{theorem}
  \label{the:actor_from_groupoid_model}
  Let \(\Gr\) and~\(\Gr[H]\) be groupoids, let \(\Reg\subseteq \Gr^0\)
  be open and \(\Gr\)\nb-invariant, and let
  \(\Bisp\colon \Gr\leftarrow \Gr\) be a groupoid correspondence that
  is regular on~\(\Reg\).  Consider a triple of continuous map
  \[
    \varrho\colon \Gr[H]^0 \to \Gr^0,\qquad
    \mu_{\Gr}\colon \Gr\times_{\s,\Gr^0,\varrho} \Gr[H]^0 \to \Gr[H],\qquad
    \mu_{\Bisp}\colon \Bisp\times_{\s,\Gr^0,\varrho } \Gr[H]^0 \to \Gr[H],
  \]
  such that~\(\mu_{\Bisp}\) is open and
  \[
    \s\circ \mu_{\Gr}(g,\omega) = \omega,\quad
    \s\circ \mu_{\Bisp}(x,\omega) = \omega,
    \quad
    \varrho\circ \rg\circ \mu_{\Gr}(g, \omega) =\rg(g),\quad
    \varrho\circ \rg\circ \mu_{\Bisp}(x, \omega) =\rg(x)
  \]
  for all \(g\in \Gr\), \(x\in \Bisp\), \(\omega \in \Gr[H]^0\) with
  \(\s(g) = \varrho(\omega)\) and \(\s(x) = \varrho(\omega)\).  Then
  define multiplication maps
  \(\Gr\times_{\s,\Gr^0,\varrho\rg} \Gr[H] \to \Gr[H]\) and
  \(\Bisp\times_{\s,\Gr^0,\varrho\rg} \Gr[H] \to \Gr[H]\) by
  \(g\cdot h \defeq \mu_{\Gr}(g, \rg(h)) \cdot h\) and
  \(x\cdot h \defeq \mu_{\Bisp}(x, \rg(h)) \cdot h\).  This data
  defines an actor \(\Mod\to \Gr[H]\) if and only if the following
  conditions hold:
  \begin{enumerate}[label=\textup{(\ref*{the:actor_from_groupoid_model}.\arabic*)},
    leftmargin=*,labelindent=0em]
  \item \label{en:actor_from_groupoid_model_1}%
    \(g_1\cdot (g_2\cdot h) = (g_1\cdot g_2)\cdot h\),
    \(g\cdot (x\cdot h) = (g\cdot x)\cdot h\), and
    \(x\cdot (g\cdot h) = (x\cdot g)\cdot h\) when
    \(g_1,g_2,g\in \Gr\), \(x\in \Bisp\), \(h\in\Gr[H]\) and these
    triples are composable;
  \item \label{en:actor_from_groupoid_model_1b}%
    \(1_{\varrho(\omega)} \cdot \omega = \omega\) for all
    \(\omega \in \Gr[H]^0\), where \(1_{\varrho(\omega)}\) is the unit
    arrow in~\(\Gr\);
  \item \label{en:actor_from_groupoid_model_2}%
    if \(x_1\cdot y_1 = x_2\cdot y_2\) then there is \(g\in \Gr\) with
    \(x_2 = x_1\cdot g\) and \(y_1 = g\cdot y_2\);
  \item \label{en:actor_from_groupoid_model_3}%
    the image of \(\rg\circ \mu_{\Bisp}\) contains
    \(\varrho^{-1}(\Reg)\);
  \item \label{en:actor_from_groupoid_model_4}%
    if \(K\subseteq \Bisp/\Gr\) is compact, then the set of all
    \(\rg(x\cdot h)\in\Gr[H]^0\) with \(x\in \Bisp\), \(h\in\Gr[H]\),
    \([x]\in K\), \(\s(x)=\rg(h)\) is closed in~\(\Gr[H]^0\).
  \end{enumerate}
  Any actor \(\Mod\to \Gr[H]\) is of this form for a unique triple of
  maps \((\varrho,\mu_{\Gr},\mu_{\Bisp})\).  The actor is proper if
  and only if \(\varrho\colon \Gr[H]^0 \to \Gr^0\) is a proper map.
  In that case, condition~\ref{en:actor_from_groupoid_model_4} is
  redundant.
\end{theorem}

\begin{proof}
  An actor \(\Mod\to \Gr[H]\) is described in
  Lemma~\ref{lem:actor_commutation} through an anchor map of the form
  \(\varrho\circ \rg\) for a continuous map
  \(\varrho\colon \Gr[H]^0\to \Gr^0\) and left multiplication maps
  \(\Gr\times_{\s,\Gr^0,\varrho \rg} \Gr[H] \to \Gr[H]\) and
  \(\Bisp\times_{\s,\Gr^0,\varrho \rg} \Gr[H] \to \Gr[H]\) that
  commute in a suitable sense with the right multiplication action
  of~\(\Gr[H]\).  Now define
  \(\mu_{\Gr}(g,\omega) \defeq g\cdot 1_\omega\) and
  \(\mu_{\Bisp}(x,\omega) \defeq x\cdot 1_\omega\).  This provides
  continuous maps with \(\s\circ \mu_{\Gr}(g,\omega) = \omega\) and
  \(\s\circ \mu_{\Bisp}(x,\omega) = \omega\) because the source map is
  invariant for the left action.  These continuous maps determine our
  left action through
  \(g\cdot h = (g\cdot 1_{\rg(h)}) \cdot h = \mu_{\Gr}(g,\rg(h))\cdot
  h\) and
  \(x\cdot h = (x\cdot 1_{\rg(h)}) \cdot h =
  \mu_{\Bisp}(x,\rg(h))\cdot h\).  The conditions
  \(\varrho\circ \rg\circ \mu_{\Gr}(g, \omega) =\rg(g)\) and
  \(\varrho\circ \rg\circ \mu_{\Bisp}(x, \omega) =\rg(x)\) are
  equivalent to \(\varrho\circ \rg(g\cdot h) =\rg(g)\) and
  \(\varrho\circ \rg(x\cdot h) =\rg(x)\), which are
  required for an action of \((\Gr,\Bisp,\Reg)\).  Since~\(\Gr[H]\) is
  \'etale, the map~\(\mu_{\Bisp}\) is open if and only if the
  multiplication map \(\Bisp\times_{\s,\varrho} \Gr[H] \to \Gr[H]\) is
  open.  Conditions
  \ref{en:actor_from_groupoid_model_1}--\ref{en:actor_from_groupoid_model_2}
  hold if and only if the left multiplication by~\(\Gr\) is a groupoid
  action and the left multiplication by~\(\Bisp\) satisfies
  \ref{en:diagram_dynamical_system_1} and
  \ref{en:diagram_dynamical_system_2}.  Conditions
  \ref{en:actor_from_groupoid_model_3} and
  \ref{en:actor_from_groupoid_model_4} are equivalent to
  \ref{en:diagram_dynamical_system_3}
  and~\ref{en:diagram_dynamical_system_4} for the action
  on~\(\Gr[H]^0\), respectively.  This is equivalent to the same
  conditions for the action on~\(\Gr[H]\) because~\(\Gr[H]^0\) is the
  orbit space for the right multiplication action of~\(\Gr[H]\) on
  itself and the sets occuring in
  \ref{en:actor_from_groupoid_model_3} and
  \ref{en:actor_from_groupoid_model_4} are \(\Gr[H]\)\nb-invariant.
  The last condition is redundant for a proper actor by
  \cite{Meyer:Diagrams_models}*{Lemma~3.7}.
\end{proof}

Actions of \((\Gr,\Bisp,\Reg)\) on a space~\(Y\) are described in
Lemma~\ref{lem:action_from_theta} using actions of bisections in
\(\Gr\) and~\(\Bisp\) by partial homeomorphisms.  In the situation of
an actor, these partial homeomorphisms are, in fact, given through
left multiplication by bisections in~\(\Gr[H]\):

\begin{theorem}
  \label{the:actor_from_groupoid_model_bisections}
  Let \(\Gr\) and~\(\Gr[H]\) be groupoids, let \(\Reg\subseteq \Gr^0\)
  be open and \(\Gr\)\nb-invariant, and let
  \(\Bisp\colon \Gr\leftarrow \Gr\) be a groupoid correspondence that
  is regular on~\(\Reg\).  Let~\(\Mod\) be the groupoid model of
  \((\Gr,\Bisp,\Reg)\).  An actor \(\Mod\to \Gr[H]\) is equivalent to
  a map
  \(\vartheta\colon \Bis(\Gr) \sqcup \Bis(\Bisp) \to \Bis(\Gr[H])\)
  such that
  \begin{enumerate}[label=\textup{(\ref*{the:actor_from_groupoid_model_bisections}.\arabic*)},
    leftmargin=*,labelindent=0em]
  \item \label{en:actor_from_groupoid_model_bisections_1}%
    \(\vartheta_V \vartheta_W = \vartheta_{V W}\) if \(V\) and~\(W\)
    are bisections in \(\Gr\) or~\(\Bisp\) and at most one of them is
    a bisection in~\(\Bisp\);
  \item\label{en:actor_from_groupoid_model_bisections_2}%
    \(\vartheta_V^* \vartheta_W = \vartheta_{\braket{V}{W}}\) if
    \(V\) and~\(W\) are bisections in~\(\Bisp\);
  \item \label{en:actor_from_groupoid_model_bisections_3}%
    \(\vartheta_{\Gr^0} = \Gr[H]^0\);
  \item \label{en:actor_from_groupoid_model_bisections_4}%
    the restriction of~\(\vartheta\) to open subsets of~\(\Gr^0\)
    commutes with arbitrary unions;
  \item \label{en:actor_from_groupoid_model_bisections_5}%
    the union of the open sets
    \(\rg(\vartheta_{\Slice})\subseteq \Gr[H]^0\) for
    \(\Slice \in \Bis(\Bisp)\) contains \(\rg(\vartheta_{\Reg})\);
  \item \label{en:actor_from_groupoid_model_bisections_6}%
    if \(\Slice\in\Bis(\Bisp)\) is precompact in~\(\Bisp\), then the
    closure of the codomain of \(\vartheta(\Slice)\) is contained in the
    union of the codomains of~\(\vartheta(W)\) for \(W\in\Bis(\Bisp)\).
  \end{enumerate}
  The actor is proper if and only if
  \(\varrho\colon \Gr[H]^0 \to \Gr^0\) is a proper map.  In that case,
  condition~\ref{en:actor_from_groupoid_model_bisections_6} is
  redundant.
\end{theorem}

\begin{proof}
  The bisections of~\(\Gr[H]\) embed into the inverse semigroup of
  partial homeomorphisms on~\(\Gr[H]\) by the left multiplication
  action.  After this embedding, the conditions in the theorem are
  exactly the conditions in Lemma~\ref{lem:action_from_theta}.  So
  these conditions ensure that there is an action of the groupoid
  model~\(\Mod\) on~\(\Gr[H]\) that induces the given partial
  homeomorphisms.  In addition, a partial homeomorphism on~\(\Gr[H]\)
  commutes with the right multiplication action of~\(\Gr[H]\) if and
  only if it is left multiplication by a bisection in~\(\Gr[H]\) by
  \cite{Buss-Exel-Meyer:InverseSemigroupActions}*{Lemma~4.10}.  Thus
  the data and conditions in the theorem describe actors
  \(\Mod \to \Gr[H]\).  Here the last condition is redundant for a
  proper actor by \cite{Meyer:Diagrams_models}*{Lemma~3.7}.
\end{proof}

\begin{remark}
  In particular, we may apply
  Theorem~\ref{the:actor_from_groupoid_model_bisections} to the
  identity actor on the groupoid model~\(\Mod\).  By the theorem, this
  corresponds to a certain map
  \(\vartheta\colon \Bis(\Gr) \sqcup \Bis(\Bisp) \to\Bis(\Mod)\).  Of
  course, this map is the standard one that is implicit in the
  construction of the groupoid model~\(\Mod\) as a transformation
  groupoid for an inverse semigroup action, where the inverse
  semigroup is generated by \(\Bis(\Gr) \sqcup \Bis(\Bisp)\).  We may
  also describe the identity actor as in
  Theorem~\ref{the:actor_from_groupoid_model}.  Here the maps
  \(\varrho\) and~\(\mu_{\Gr}\) combine to an actor from~\(\Gr\)
  to~\(\Mod\).  This actor is implicit in the universal property
  of~\(\Mod\) because an action of~\(\Mod\) on a space is equivalent
  to an action of~\(\Gr\) together with some extra data.
\end{remark}

We now specialise to the case when the topological space~\(\Gr^0\) is
discrete.  Then~\(\{x\}\) for an element of \(\Gr\) or~\(\Bisp\) is a
bisection.  We abbreviate \(\vartheta_x = \vartheta_{\{x\}}\) for
\(x\in \Gr\sqcup \Bisp\) and get the following corollary:

\begin{corollary}
  \label{cor:actor_from_discrete_groupoid_model_bisections}
  Let \(\Gr\) and~\(\Gr[H]\) be groupoids, let \(\Reg\subseteq \Gr^0\)
  be open and \(\Gr\)\nb-invariant, and let
  \(\Bisp\colon \Gr\leftarrow \Gr\) be a groupoid correspondence that
  is regular on~\(\Reg\).  Assume that~\(\Gr\) is discrete.
  Let~\(\Mod\) be the groupoid model of \((\Gr,\Bisp,\Reg)\).  An
  actor \(\Mod\to \Gr[H]\) is equivalent to a map
  \(\vartheta\colon \Gr \sqcup \Bisp \to \Bis(\Gr[H])\)
  such that
  \begin{enumerate}[label=\textup{(\ref*{cor:actor_from_discrete_groupoid_model_bisections}.\arabic*)},
    leftmargin=*,labelindent=0em]
  \item \label{en:actor_from_discrete_groupoid_model_bisections_1}%
    \(\vartheta_x \vartheta_y = \vartheta_{x y}\) if
    \(x,y\in\Gr\sqcup \Bisp\) and \(x\in\Gr\) or \(y\in\Gr\);
  \item \label{en:actor_from_discrete_groupoid_model_bisections_2}%
    \(\vartheta_x^* \vartheta_x = \vartheta_{\s(x)}\) for all
    \(x\in\Bisp\), and \(\vartheta_x^* \vartheta_y=\emptyset\) for
    \(x,y\in\Bisp\) in different right \(\Gr\)\nb-orbits;
  \item \label{en:actor_from_discrete_groupoid_model_bisections_3}%
    the domains of~\(\vartheta_x\) for \(x\in\Gr^0\)
    cover~\(\Gr[H]^0\);
  \item \label{en:actor_from_discrete_groupoid_model_bisections_5}%
    if \(x\in\Reg\), then the domain of~\(\vartheta_x\) is contained
    in the union of the codomains of~\(\vartheta_y\) for
    \(y\in \Bisp\);
  \item \label{en:actor_from_discrete_groupoid_model_bisections_6}%
    the codomain of~\(\vartheta_x\) is closed for all \(x\in\Bisp\).
  \end{enumerate}
  The actor is proper if and only if \(\vartheta_x\) for
  \(x\in \Gr^0\) is a compact subset of~\(\Gr[H]^0\).  If the actor is
  proper, then
  condition~\ref{en:actor_from_discrete_groupoid_model_bisections_6} is
  redundant.
\end{corollary}

\begin{proof}
  In the situation of
  Corollary~\ref{cor:actor_from_discrete_groupoid_model_bisections},
  \ref{en:actor_from_discrete_groupoid_model_bisections_1}
  and~\ref{en:actor_from_discrete_groupoid_model_bisections_2} imply
  \(\vartheta_x^* \vartheta_{x g} = \vartheta_g\) for all
  \(x\in\Bisp\), \(g\in\Gr\), so
  that~\ref{en:actor_from_discrete_groupoid_model_bisections_2} allows
  to simplify \(\vartheta_x^* \vartheta_y\) for all \(x,y\in\Bisp\) as
  in Condition~\ref{en:actor_from_groupoid_model_bisections_2}.  The
  bisections~\(\vartheta_x\) for \(x\in\Gr^0\) are idempotent.  Thus
  they are contained in the unit bisection \(\Gr[H]^0\).  If the actor
  is proper, then the domain of
  \(\vartheta_x^* \vartheta_x = \vartheta_{\s(x)}\) is compact for all
  \(x\in\Bisp\).  Then the domain of \(\vartheta_x\vartheta_x^*\) is
  compact as well because~\(\vartheta_x\) implements a homeomorphism
  between these two subsets.  Thus
  condition~\ref{en:actor_from_discrete_groupoid_model_bisections_6}
  is redundant in the proper case as asserted.
\end{proof}

\begin{corollary}
  \label{cor:dynamical_CK_as_actor}
  Let \(\Gr=V\) be a set, made a groupoid with only identity arrows.
  Let \(\rg,\s\colon E\rightrightarrows V\) be a directed graph,
  viewed as a groupoid correspondence
  \(\Bisp\colon \Gr\leftarrow \Gr\).  Let \(\Reg \subseteq V\) be the
  set of regular vertices of the graph.  Let~\(\Mod\) be the groupoid
  model of \((\Gr,\Bisp,\Reg)\) and let~\(\Gr[H]\) be another \'etale
  locally compact groupoid.  Then there is a canonical bijection
  between proper actors from~\(\Mod\) to~\(\Gr[H]\) and nondegenerate
  dynamical Cuntz--Krieger families in~\(\Gr[H]\).
\end{corollary}

\begin{corollary}
  Assume that two graph \(\Cst\)\nb-algebras have no subquotients
  isomorphic to \(\Cont(\T)\).  If a dynamical Cuntz--Krieger family
  induces an isomorphism between them, then the inverse is also
  induced by a dynamical Cuntz--Krieger family.
\end{corollary}

\begin{proof}
  Nondegenerate dynamical Cuntz--Krieger families are equivalent to
  actors between the groupoid models of the graphs.  This groupoid is
  always Hausdorff, and it is effective if there are no subquotients
  isomorphic to \(\Cont(\T)\).  Then
  Theorem~\ref{the:Cartan_isomorphism} implies that if an isomorphism
  of groupoid \(\Cst\)\nb-algebras is induced by a groupoid actor,
  then so is its inverse.  A dynamical Cuntz--Krieger family that
  induces an isomorphism on the groupoid \(\Cst\)\nb-algebras is
  automatically nondegenerate.
\end{proof}

Let now \(\rg_j,\s_j\colon E_j\rightrightarrows V_j\) for \(j=1,2\) be
two directed graphs.  Let \((\Gr_j,\Bisp_j,\Reg_j)\) be associated to
them as in Corollary~\ref{cor:dynamical_CK_as_actor} and
let~\(\Mod_j\) be the resulting groupoid models.  Let~\(\Gr[H]\) be a
third \'etale locally compact groupoid.  Consider a nondegenerate
dynamical Cuntz--Krieger family for the graph
\(\rg_1,\s_1\colon E_1\rightrightarrows V_1\) in~\(\Mod_2\) and a
nondegenerate dynamical Cuntz--Krieger family for the graph
\(\rg_2,\s_2\colon E_2\rightrightarrows V_2\) in~\(\Gr[H]\).  Then we
may identify these nondegenerate dynamical Cuntz--Krieger families with
actors \(\Mod_1 \to \Mod_2\) and \(\Mod_2 \to \Gr[H]\).  The latter
two may be composed to an actor \(\Mod_1 \to \Gr[H]\), which in turn
corresponds to a nondegenerate dynamical Cuntz--Krieger family for the
graph \(\rg_1,\s_1\colon E_1\rightrightarrows V_1\) in~\(\Gr[H]\).  In
particular, if~\(\Gr[H]\) is the groupoid model of another graph, then
this gives a composition for nondegenerate dynamical Cuntz--Krieger
families in groupoid models of graphs.  This gives rise to a category
because actors form a category.  In this sense, we may interpret
nondegenerate dynamical Cuntz--Krieger families as morphisms between
the graphs.

\section{Comparison to relation morphisms of directed graphs}

In this section, we show that the \Star{}homomorphism between two
graph \(\Cst\)\nb-algebras defined by an admissible relation morphism
as in~\cite{Castro-DAndrea-Hajac:Relation_morphisms} also comes from a
possibly degenerate dynamical Cuntz--Krieger family.  Here
\emph{degenerate} means that we drop the condition
\(\Gr[H]^0 = \bigsqcup_{v\in V} \Omega_v\) that is needed for a
dynamical Cuntz--Krieger family to correspond to an actor.

Let \(R\subseteq X\times Y\) be a relation from~\(X\) to~\(Y\).  We
call \(R(x)\defeq\setgiven{y\in Y}{(x,y)\in R}\) for \(x\in X\) the
\emph{image} of \(x\in X\) and
\(R^{-1}(y)\defeq\setgiven{x\in Y}{(x,y)\in R}\) the \emph{preimage}
of \(y\in Y\).

Let \(\Gamma=(\rg,\s\colon E\rightrightarrows V)\) be a graph.  Let
\(FP(\Gamma)\) denote the set of finite paths in~\(\Gamma\) and let
\(\rg_{P\Gamma}\) and~\(\s_{P\Gamma}\) denote the extensions of
\(\rg\) and~\(\s\) to finite paths.  Let~\(\Reg\) be the set of
regular vertices of~\(\Gamma\).  For \(x,x'\in FP(\Gamma)\), we write
\(x\preceq x'\) if \(x'=xy\) for some \(y\in FP(\Gamma)\), and
\(x\sim x'\) if \(x\preceq x'\) or \(x'\preceq x\).

\begin{definition}[\cite{Castro-DAndrea-Hajac:Relation_morphisms}*{Definition~3.1}]
  \label{def:relation.morphism}
  Let \(\Gamma_j=(\rg_j,\s_j\colon E_j\rightrightarrows V_j)\) for
  \(j=1,2\) be graphs.  A \emph{relation morphism}
  \(R\colon \Gamma_1\to \Gamma_2\) is a relation
  \(R\subseteq FP(\Gamma_1) \times FP(\Gamma_2)\) such that
  \begin{enumerate}
  \item if \((x,y)\in R\) and \(y\in V_2\), then \(x\in V_1\) (that is
    \(R^{-1}(V_2)\subseteq V_1\));
  \item if \((x,y)\in R\), then
    \((\s_{P\Gamma_1}(x), \s_{P\Gamma_2}(y))\in R\) (source
    preserving);
  \item if \((x,y)\in R\), then
    \((\rg_{P\Gamma_1}(x), \rg_{P\Gamma_2}(y))\in R\) (range
    preserving).
  \end{enumerate}
\end{definition}

\begin{definition}[\cite{Castro-DAndrea-Hajac:Relation_morphisms}*{Definition~6.5}]
  \label{def:admissible.relation}
  Let \(\Gamma_1\) and~\(\Gamma_2\) be graphs.  A relation morphism
  \(R\colon\Gamma_1\to \Gamma_2\) is said to be \emph{admissible} if:
  \begin{enumerate}
  \item it is multiplicative, that is, if \((x,y),(x',y')\in R\),
    \(\s_{P\Gamma_1}(x)=\rg_{P\Gamma_1}(x')\),
    \(\s_{P\Gamma_2}(y)=\rg_{P\Gamma_2}(y')\), then
    \((x x',y y')\in R\);
  \item it is decomposable, that is, if \(y,y'\in FP(\Gamma_2)\),
    \(\s_{P\Gamma_2}(y)=\rg_{P\Gamma_2}(y')\),
    \((\overline{x},yy')\in R\), then there are \(x\in R^{-1}(y)\) and
    \(x'\in R^{-1}(y')\) with \(xx'=\overline{x}\);
  \item it is proper, that is, \(\abs{R^{-1}(y)}<\infty\) for all
    \(y\in FP(\Gamma_2)\);
  \item it is vertex disjoint, that is,
    \(R^{-1}(v)\cap R^{-1}(v')=\emptyset\) for every \(v,v'\in V_2\);
  \item it is source bijective, that is, the restriction
    \(\s_{P\Gamma_1}\colon R^{-1}(f)\to R^{-1}(\s_2(f))\) is bijective
    for all \(f\in E_2\);
  \item it is monotone, that is, if \((x,f),(x',f')\in R\),
    \(f,f'\in E_2\), \(x\preceq x'\), then \((x,f)=(x',f')\);
  \item it is regular, that is, if \(v\in V_2\) is regular,
    \(u\in R^{-1}(v)\) and \(x\in \rg_{P\Gamma_1}^{-1}(u)\), then there
    are \((x',f)\in R\), \(f\in E_2\) with \(\rg_2(f)=v\) and
    \(x\sim x'\).
  \end{enumerate}
\end{definition}

\begin{remark}
  The convention for paths in this paper is different from the one
  in~\cite{Castro-DAndrea-Hajac:Relation_morphisms}.  We made the
  suitable adaptations in Definitions \ref{def:relation.morphism}
  and~\ref{def:admissible.relation}.
\end{remark}

Given a graph~\(\Gamma\), let~\(\partial \Gamma\) be its boundary path
space and let~\(G_{\Gamma}\) be its boundary path groupoid.  For each
\(x,y\in FP(\Gamma)\) with \(\s(x)=\s(y)\), we define
\begin{align*}
  Z_x
  &\defeq \setgiven{z\in\partial \Gamma}{z=x z'
    \text{ for some } z'\in\partial\Gamma},\\
  Z(x,y)
  &\defeq \setgiven{(x z,\abs{x}-\abs{y},y z)\in G_{\Gamma}}
    {z\in\partial\Gamma,\ \rg(z)=\s(x)}.
\end{align*}

\begin{theorem}
  \label{thm:relation.to.CK}
  Let \(\Gamma_1\) and~\(\Gamma_2\) be graphs, and let
  \(R\colon\Gamma_1\to \Gamma_2\) be an admissible relation morphism.
  For each \(v\in V_2\), set
  \[
    \Omega_v \defeq \bigcup_{u\in R^{-1}(v)} Z_u,
  \]
  and for each \(f\in E_2\), set
  \[
    T_f\defeq \bigcup_{x\in R^{-1}(f)} Z(x,\s(x)).
  \]
  The families \((\Omega_v)_{v\in V_2}\) and \((T_f)_{f\in E_2}\) form
  a possibly degenerate dynamical Cuntz--Krieger
  \(\Gamma_2\)\nb-family in~\(G_{\Gamma_1}\).  The dynamical
  Cuntz--Krieger family is nondegenerate if and only if
  \(R^{-1}(V_2)=V_1\).  In addition,
  \begin{enumerate}[label=\textup{(\ref*{thm:relation.to.CK}.\arabic*)},
    leftmargin=*,labelindent=0em]
  \item \label{it:dCKtoR1}%
    for all \(u\in V_1\) and \(v\in V_2\), either
    \(Z_u\Omega_v=\emptyset\) or \(Z_u\Omega_v=Z_u\),
  \item \label{it:dCKtoR2}%
    for all \(u\in V_1\) and \(f\in E_2\), either \(T_fZ_u=\emptyset\)
    or there is \(x\in FP(\Gamma_1)\) such that \(T_fZ_u=Z(x,s(x))\).
  \end{enumerate}
\end{theorem}

\begin{proof}
  Since~\(R\) is proper, the unions defining each \(\Omega_v\)
  and~\(T_f\) are finite.  So these are compact-open subsets.  We
  prove that~\(T_f\) is a bisection.  Let \(x,x'\in R^{-1}(f)\) and
  \(z,z'\in\partial\Gamma\) be such that
  \(\gamma\defeq(x z,\abs{x},z)\in Z(x,\s(x))\) and
  \(\gamma'\defeq(x'z',\abs{x'},z')\in Z(x',\s(x'))\).  Note that
  \((\s(x),\s(f)),(\s(x'),\s(f))\in R\).  First, assume that \(z=z'\).
  Then \(\s(x)=\s(x')\).  Since~\(R\) is source-bijective, this
  implies \(x=x'\) and then \(\gamma=\gamma'\).  Now assume that
  \(x z=x' z'\), so that \(x\sim x'\).  Since~\(R\) is monotone, this
  implies \(x=x'\), which further implies \(\gamma=\gamma'\).

  We now prove that the families satisfy the four conditions of the
  definition of a dynamical Cuntz--Krieger family.  The final claim
  about nondegeneracy is trivial.

  We prove the condition~\ref{en:dCK_1}.  Let \(v,v'\in V_2\) be such
  that \(v\neq v'\).  Since~\(R\) is vertex disjoint,
  \(R^{-1}(v)\cap R^{-1}(v')=\emptyset\).  Note also that
  \(Z_u\cap Z_{u'}=\emptyset\) for \(u,u'\in V_1\) with \(u\neq u'\).
  It follows that \(\Omega_v\cap \Omega_{v'}=\emptyset\).

  We prove the condition~\ref{en:dCK_2}.  Let \(f\in E_2\),
  \(\gamma=(xz,\abs{x},z)\in Z(x,\s(x))\subseteq T_f\) for some
  \(x\in R^{-1}(f)\), and \(w\in Z_u\subseteq \Omega_{\s(f)}\) for
  some \(u\in R^{-1}(\s(f))\).  Note that
  \((\s(x),\s(f)),(\rg(x),\rg(f))\in R\).  Now
  \begin{align*}
    \s(\gamma)
    &= z\in Z_{\s(x)}\subseteq \Omega_{\s(f)},\\
    \rg(\gamma)
    &= xz\in Z_{\rg(x)}\subseteq \Omega_{\rg(f)}.
  \end{align*}
  Since~\(R\) is source bijective and \(u\in R^{-1}(\s(f))\), there is
  \(x'\in R^{-1}(f)\) with \(\s(x)=u=\rg(w)\).  Then
  \(\gamma'\defeq(x' w,\abs{x'},w)\in T_f\) and \(\s(\gamma')=w\).

  We prove the condition~\ref{en:dCK_3}.  Let \(f,f'\in E^2\),
  \(\gamma=(xz,\abs{x},z)\in Z(x,\s(x))\subseteq T_f\) and
  \(\gamma'=(x'z',\abs{x'},z')\in Z(x',\s(x'))\subseteq T_{f'}\),
  where \((x,f),(x',f')\in R\) and \(z,z'\in\partial \Gamma_1\).
  Assume that \(xz=\rg(\gamma)=\rg(\gamma')=x'z'\).  In this case
  \(x\sim x'\).  Since~\(R\) is monotone, it follows that \(f=f'\).

  We prove the condition~\ref{en:dCK_4}.  Let~\(v\) be a regular
  vertex in~\(V_2\).  All that it remains to prove is that
  \(\Omega_v\subseteq \bigsqcup_{\rg(f)=v} \rg(T_f)\).  Let
  \(z\in\Omega_v\) and note that \((\rg(z),v)\in R\).

  Suppose first that \(\abs{z}<\infty\).  Since~\(v\) is regular,
  there are \(f\in \rg^{-1}(v)\) and \((x,f)\in R\) with \(x\sim z\).
  Since~\(\rg(z)\) is not a regular vertex of~\(V_1\), by
  \cite{Castro-DAndrea-Hajac:Relation_morphisms}*{Lemma~6.2},
  \(x\preceq z\).  Hence there is \(z'\in\partial\Gamma_1\) with
  \(z=xz'\).  Then
  \(\gamma\defeq(z,\abs{x},z')\in Z(x,\s(x))\subseteq T_f\) and
  \(\rg(\gamma)=z\).

  Suppose now that \(\abs{z}=\infty\).  Since~\(R\) is proper
  and~\(v\) is regular, there is an initial segment~\(w\) of~\(z\)
  with \(\abs{w}>\abs{x}\) for all \(x\in R^{-1}(\rg^{-1}(v))\).
  Since~\(R\) is regular, there are \(f\in \rg^{-1}(v)\) and
  \((x,f)\in R\) with \(x\sim w\).  Since \(\abs{x}<\abs{w}\), it
  follows that \(x\preceq w\).  This implies that there is
  \(z'\in\partial\Gamma_1\) with \(z=xz'\).  As above, if we let
  \(\gamma\defeq(z,\abs{x},z')\in Z(x,\s(x))\subseteq T_f\), then
  \(\rg(\gamma)=z\).
  
  Finally, \ref{it:dCKtoR1} follows because~$R$ is vertex disjoint and
  $Z_uZ_{u'}=\emptyset$ if $u,u'\in V_1$ and $u\neq u'$.  As
  for~\ref{it:dCKtoR2}, if \(u\in V_1\) and \(f\in E_2\) are such that
  \(T_fZ_u\neq\emptyset\), then there is $x\in R^{-1}(f)$ with
  $s(x)=u$.  The element~$x$ is unique because~$R$ is source bijective.
  Hence \(T_fZ_u=Z(x,s(x))Z_u=Z(x,s(x))\).
\end{proof}

\begin{theorem}
  \label{thm:from.dCKtoR}
  Let \(\Gamma_1\) and~\(\Gamma_2\) be graphs, and let
  \((\Omega_v)_{v\in E_2}\) and \((T_f)_{f\in E_2}\) be a dynamical
  Cuntz--Krieger \(\Gamma_2\)\nb-family in~\(G_{\Gamma_1}\).  Assume
  that the family satisfies \ref{it:dCKtoR1} and \ref{it:dCKtoR2}.
  Then there is an admissible relation
  \(R\colon \Gamma_1\to \Gamma_2\) such that the family given by
  Theorem~\textup{\ref{thm:relation.to.CK}} for~\(R\) coincides with
  \((\Omega_v)_{v\in E_2}\) and \((T_f)_{f\in E_2}\).
\end{theorem}

\begin{proof}
  Let
  \begin{align*}
    R^0&\defeq \setgiven{(u,v)\in V_1\times V_2}{Z_u\Omega_v=Z_u}\\
    R^1&\defeq
    \setgiven{(x,f)\in FP(\Gamma_1)\times E_2}{T_fZ(s(x),x)=Z(x,x)}.
  \end{align*}
  We will use
  \cite{Castro-DAndrea-Hajac:Relation_morphisms}*{Lemma~3.11} to build
  a multiplicative and decomposable relation morphism~\(R\).  Let
  \((x,f)\in R^1\) be given.  First, we prove that
  \((s(x),s(f))\in R^0\).  Note that
  \[T_f Z_{s(x)}=T_f Z(s(x),x)Z(x,s(x))=Z(x,x)Z(x,s(x))=Z(x,s(x)).\]
  Hence
  \begin{align*}
    Z_{s(x)}\Omega_{s(f)}&=Z_{s(x)}T_f^{-1}T_fZ_{s(x)}\\
                         &=(T_fZ_{s(x)})^{-1}T_fZ_{s(x)}\\
                         &=Z(x,s(x))^{-1}Z(x,s(x))\\
                         &=Z_{s(x)}.
  \end{align*}
  The definition of~\(R^0\) implies that \((s(x),s(f))\in R^0\).  Now
  we prove that \((r(x),r(f))\in R^0\).  The equality
  \(T_fZ(s(x),x)=Z(x,x)\) implies that \((xy,|x|,y)\in T_f\) for every
  \(y\in Z_{s(x)}\).  Since \(Z_{s(x)}\neq\emptyset\), there is
  \(y\in Z_{s(x)}\).  Then
  \[
    (xy,0,xy)\in Z_{r(x)}T_fT_f^{-1}\subseteq
    Z_{r(x)}\Omega_{r(f)}.
  \]
  This implies \(Z_{r(x)}\Omega_{r(f)}\neq\emptyset\) and by
  hypothesis \(Z_{r(x)}\Omega_{r(f)}=Z_{r(x)}\).  So1
  \((r(x),r(f))\in R^0\).
  
  The next step is to show that~\(R\) obtained from
  \cite{Castro-DAndrea-Hajac:Relation_morphisms}*{Lemma~3.11} is
  admissible.  That~\(R\) is vertex disjoint follows from
  \ref{en:dCK_1}.
  
  Next we show that~\(R\) is monotone.  Let \((x,f),(x',f')\in R^1\)
  be such that \(x\preceq x'\).  The latter condition implies that
  \(Z(x',x')\subseteq Z(x,x)\) and hence
  \[
    Z(x',s(x'))T_{f'}^{-1}T_fZ(s(x),x)
    =Z(x',x')Z(x,x)
    =Z(x',x')\neq\emptyset.
  \]
  Now~\ref{en:dCK_3} implies \(f=f'\).  As a consequence,
  \[
    Z(s(x'),s(x'))T_{f}^{-1}T_fZ(s(x),s(x))\neq \emptyset
  \]
  and then \(s(x)=s(x')\).  If \(y\in Z_{s(x)}\), then
  \((xy,|x|,y),(x'y,|x'|,y)\in T_f\).  Since~\(T_f\) is a bisection,
  \(x=x'\).  Hence~\(R\) is monotone.
  
  Let us now prove that~\(R\) is source bijective.  That
  \(s\colon R^{-1}(f)\to R^{-1}(s(f))\) is injective follows as above.
  If \(x,x'\in R^{-1}(f)\) are such that \(s(x)=s(x')\), by choosing
  \(y\in Z_{s(x)}\) and using that~\(T_f\) is a bisection, we conclude
  that \(x=x'\).  Let \(u\in R^{-1}(s(f))\) so that
  \(Z_u\Omega_{s(f)}\neq\emptyset\).  Now \ref{en:dCK_2} implies
  \(T_fZ_u\neq\emptyset\).  By hypothesis, there is
  \(x\in FP(\Gamma_1)\) with \(T_fZ_u=Z(x,s(x))\).  Note that this
  condition implies that \(s(x)=u\).  Moreover,
  \begin{equation}
    \label{eq:dCKtoR.eq1}
    T_fZ(s(x),x)=T_fZ_uZ(s(x),x)=Z(x,s(x))Z(s(x),x)=Z(x,x)
  \end{equation}
  and hence \(x\in R^{-1}(f)\).  It follows that
  \(s\colon R^{-1}(f)\to R^{-1}(s(f))\) is surjective.
  
  For properness, note that if \(v\in V_2\), then \(R^{-1}(v)\) is
  finite because~\(\Omega_v\) is compact and
  \(\partial E_1=\bigsqcup_{u\in V_1}Z_u\) and~\(Z_u\) is open for
  every \(u\in V_1\).  Since we have already proved that~\(R\) is
  source bijective, we see also that~\(R^{-1}(f)\) is finite for every
  \(f\in E^2\).  Finally, decomposability and multiplicativity implies
  that \(R^{-1}(y)\) is finite for every \(y\in FP(\Gamma_2)\).
  Hence~\(R\) is proper.
  
  Next, \ref{it:dCKtoR1} and the definition of~\(R\) imply
  \[
    \Omega_v = \bigcup_{u\in R^{-1}(v)} Z_u.
  \]
  Moreover, \ref{it:dCKtoR1}, \eqref{eq:dCKtoR.eq1} and the definition
  of~\(R\) imply
  \begin{equation}
    \label{eq:dCKtoR.eq2}
    T_f= T_f\Omega_{s(f)}
    = T_f\bigsqcup_{u\in R^{-1}(s(f))}Z_u
    = \bigsqcup_{x\in R^{-1}(f)}Z(x,s(x)).
  \end{equation}
  
  So the dynamical Cuntz--Krieger family associated to $R$ in
  Theorem~\ref{thm:relation.to.CK} is \((\Omega, T)\).  It remains to
  prove that \(R\) is regular.  Let \(v\in V_2\) be regular.  Let also
  \(u\in R^{-1}(v)\) and \(x\in \rg_{P\Gamma_1}^{-1}(u)\).  The
  definition of~\(R\) implies \(Z_u\Omega_v=Z_u\).  Let
  \(y\in Z_{\s(x)}\) and note that \((xy,0,xy)\in Z_u\).  By
  \ref{en:dCK_4}, there is \(f\in E_2\) with \(r(f)=v\) and
  \((xy,0,xy)\in Z_uT_fT_f^{-1}\).  By \eqref{eq:dCKtoR.eq2}, we then
  find \(x'\in R^{-1}(f)\) such that
  \[
    (xy,0,xy)\in Z_uZ(x',s(x'))Z(s(x'),x')=Z_uZ(x',x').
  \]
  This implies that \((xy,0,xy)\in Z(x',x')\), and then \(x\sim x'\)
  follows.  Hence~\(R\) is regular.
\end{proof}

\begin{example}
  \label{exa:concrete_inverse}
  For \(i=1,2\), consider the graphs \(\Gamma_i=(V_i,E_i)\), where
  \(V_1=\{u_1,u_2,u_3\}\), \(E_1=\{e_1,e_2\}\), \(V_2=\{v_1,v_2\}\)
  and \(E_2=\{f_1,f_2\}\).  The source and range maps are given by
  \(s(e_i)=u_i\), \(r(e_i)=u_{i+1}\), \(s(f_i)=v_1\) and
  \(r(f_i)=v_2\) for \(i=1,2\).  The C*-algebra of both graphs are
  isomorphic to \(\Mat_3(\C)\).
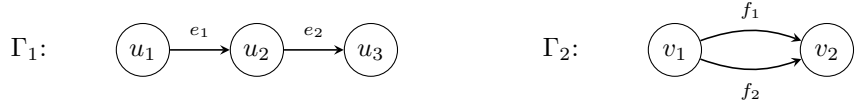
  \[
    \begin{tikzpicture}[
      vertex/.style={circle, draw, minimum size=7mm, inner sep=0pt},
      edge/.style={->, >=stealth, semithick},
      corr/.style={|->}
      ]
      \node at (-1, 1) {$\Gamma_1$:};
      \node[vertex] (u1) at (0.5,1) {$u_1$};
      \node[vertex] (u2) at (2,1) {$u_2$};
      \node[vertex] (u3) at (3.5,1) {$u_3$};
      \draw[edge] (u1) -- node[above, font=\scriptsize] {$e_1$} (u2);
      \draw[edge] (u2) -- node[above, font=\scriptsize] {$e_2$} (u3);
      
      \node at (6, 1) {$\Gamma_2$:};
      \node[vertex] (v1) at (7.5,1) {$v_1$};
      \node[vertex] (v2) at (9.5,1) {$v_2$};
      \draw[edge] (v1) to[bend left=20] node[above, font=\scriptsize] {$f_1$} (v2);
      \draw[edge] (v1) to[bend right=20] node[below, font=\scriptsize] {$f_2$} (v2);
    \end{tikzpicture}
  \]
  By \cite{Castro-DAndrea-Hajac:Relation_morphisms}*{Example~7.1}, the
  relation
  \[
    R=\{(u_1,v_1),(u_2,v_2),(u_3,v_2),(e_2e_1,f_2),(e_1,f_1)\}
  \]
  is admissible.  Then Theorem~\ref{thm:relation.to.CK} produces the
  following dynamical Cuntz--Krieger \(\Gamma_2\)\nb-family
  in~\(G_{\Gamma_1}\):
  \[
    \Omega_{v_1} = Z_{u_1},\quad
    \Omega_{v_2} = Z_{u_2} \sqcup Z_{u_3},\qquad
    T_{f_1} = Z(e_1,u_1),\quad
    T_{f_2} = Z(e_2 e_1,u_1).
  \]
  This induces an actor \(G_{\Gamma_2}\to G_{\Gamma_1}\) and then a
  \Star{}homomorphism \(\Mat_3(\C) \to \Mat_3(\C)\) by
  Proposition~\ref{actor_to_Cstar}.  The latter is not zero and hence
  an isomorphism because~\(\Mat_3(\C)\) is simple.
  Theorem~\ref{the:Cartan_isomorphism} implies that the underlying
  actor is an isomorphism of groupoids.  Some standard computations
  show that the inverse is given by the following dynamical
  Cuntz--Krieger \(\Gamma_1\)\nb-family in~\(G_{\Gamma_2}\):
  \begin{gather*}
    \Omega_{u_1}=Z_{v_1},\quad
    \Omega_{u_2}=Z_{f_1},\quad
    \Omega_{u_3}=Z_{f_2},\\
    T_{e_1}=Z(f_1,s(f_1)),\quad
    T_{e_2}=Z(f_2,f_1).
  \end{gather*}
  Since $\Omega_{u_2}Z_{v_2}=Z_{f_1}$ does not come from a vertex,
  this family does not come from an admissible relation by
  Theorem~\ref{thm:relation.to.CK}.
\end{example}


\bib{Albandik:Thesis}{thesis}{
  author={Albandik, Suliman},
  title={A colimit construction for groupoids},
  institution={Georg-August-Universit\"at G\"ottingen},
  date={2015},
  type={phdthesis},
  doi={10.53846/goediss-5785},
}

\bib{Antunes:Thesis}{thesis}{
  author={Antunes, Celso},
  title={C*-algebras of relatively proper groupoid correspondences},
  institution={Georg-August-Universit\"at G\"ottingen},
  type={phdthesis},
  date={2024},
  doi={10.53846/goediss-10930},
}

\bib{Antunes-Ko-Meyer:Groupoid_correspondences}{article}{
  author={Antunes, Celso},
  author={Ko, Joanna},
  author={Meyer, Ralf},
  title={The bicategory of groupoid correspondences},
  journal={New York J. Math.},
  date={2022},
  pages={1329--1364},
  volume={28},
  eprint={https://nyjm.albany.edu/j/2022/28-56.html},
  issn={1076-9803},
  review={\MR {4489452}},
}

\bib{Buneci-Stachura:Morphisms_groupoids}{article}{
  author={Buneci, M\u {a}d\u {a}lina Roxana},
  author={Stachura, Piotr},
  title={Morphisms of locally compact groupoids endowed with Haar systems},
  status={eprint},
  date={2005},
  note={\arxiv {math/0511613}},
}

\bib{Buss-Exel-Meyer:InverseSemigroupActions}{article}{
  author={Buss, Alcides},
  author={Exel, Ruy},
  author={Meyer, Ralf},
  title={Inverse semigroup actions as groupoid actions},
  journal={Semigroup Forum},
  volume={85},
  date={2012},
  number={2},
  pages={227--243},
  issn={0037-1912},
  doi={10.1007/s00233-012-9418-y},
  review={\MR {2969047}},
}

\bib{Clark-Exel-Pardo-Sims-Starling:Simplicity_non-Hausdorff}{article}{
  author={Clark, Lisa Orloff},
  author={Exel, Ruy},
  author={Pardo, Enrique},
  author={Sims, Aidan},
  author={Starling, Charles},
  title={Simplicity of algebras associated to non-Hausdorff groupoids},
  journal={Trans. Amer. Math. Soc.},
  volume={372},
  date={2019},
  number={5},
  pages={3669--3712},
  issn={0002-9947},
  review={\MR {3988622}},
  doi={10.1090/tran/7840},
}

\bib{Castro-DAndrea-Hajac:Relation_morphisms}{article}{
  title={Relation morphisms of directed graphs},
  author={de Castro, Gilles G.},
  author={D'Andrea, Francesco},
  author={Hajac, Piotr M.},
  year={2025},
  status={eprint},
  note={\arxiv {2503.23343}},
}

\bib{Katsura:class_I}{article}{
  author={Katsura, Takeshi},
  title={A class of $C^*$\nobreakdash -algebras generalizing both graph algebras and homeomorphism $C^*$\nobreakdash -algebras. I. Fundamental results},
  journal={Trans. Amer. Math. Soc.},
  volume={356},
  date={2004},
  number={11},
  pages={4287--4322},
  issn={0002-9947},
  review={\MR {2067120}},
  doi={10.1090/S0002-9947-04-03636-0},
}

\bib{Kumjian-Pask-Raeburn-Renault:Graphs}{article}{
  author={Kumjian, Alex},
  author={Pask, David},
  author={Raeburn, Iain},
  author={Renault, Jean},
  title={Graphs, groupoids, and Cuntz--Krieger algebras},
  journal={J. Funct. Anal.},
  volume={144},
  date={1997},
  number={2},
  pages={505--541},
  issn={0022-1236},
  review={\MR {1432596}},
  doi={10.1006/jfan.1996.3001},
}

\bib{Kwasniewski-Meyer:Cartan}{article}{
  author={Kwa\'sniewski, Bartosz Kosma},
  author={Meyer, Ralf},
  title={Noncommutative Cartan \(\textup {C}^*\)\nobreakdash -subalgebras},
  journal={Trans. Amer. Math. Soc.},
  volume={373},
  date={2020},
  number={12},
  pages={8697--8724},
  issn={0002-9947},
  review={\MR {4177273}},
  doi={10.1090/tran/8174},
}

\bib{Meyer:Diagrams_models}{article}{
  author={Meyer, Ralf},
  title={Groupoid models for diagrams of groupoid correspondences},
  note={\arxiv {2203.12068}},
  status={eprint},
  date={2022},
}

\bib{Meyer:Groupoid_models_relative}{article}{
  author={Meyer, Ralf},
  title={Groupoid models for relative Cuntz--Pimsner algebras of groupoid correspondences},
  status={eprint},
  date={2025},
  doi={10.48550/arXiv.2506.19569},
}

\bib{Meyer-Zhu:Groupoids}{article}{
  author={Meyer, Ralf},
  author={Zhu, {Ch}enchang},
  title={Groupoids in categories with pretopology},
  journal={Theory Appl. Categ.},
  volume={30},
  date={2015},
  pages={1906--1998},
  issn={1201-561X},
  eprint={http://www.tac.mta.ca/tac/volumes/30/55/30-55abs.html},
  review={\MR {3438234}},
}

\bib{Nekrashevych:Self-similar}{book}{
  author={Nekrashevych, Volodymyr},
  title={Self-similar groups},
  series={Mathematical Surveys and Monographs},
  volume={117},
  publisher={American Mathematical Society, Providence, RI},
  date={2005},
  pages={xii+231},
  isbn={0-8218-3831-8},
  review={\MR {2162164}},
  doi={10.1090/surv/117},
}

\bib{Nekrashevych:Cstar_selfsimilar}{article}{
  author={Nekrashevych, Volodymyr},
  title={$C^*$\nobreakdash -algebras and self-similar groups},
  journal={J. Reine Angew. Math.},
  volume={630},
  date={2009},
  pages={59--123},
  issn={0075-4102},
  review={\MR {2526786}},
  doi={10.1515/CRELLE.2009.035},
}

\bib{Paterson:Graph_groupoids}{article}{
  author={Paterson, Alan L. T.},
  title={Graph inverse semigroups, groupoids and their $C^*$\nobreakdash -algebras},
  journal={J. Operator Theory},
  volume={48},
  date={2002},
  number={3, suppl.},
  pages={645--662},
  issn={0379-4024},
  review={\MR {1962477}},
  eprint={http://www.theta.ro/jot/archive/2002-048-003/2002-048-003-012.html},
}

\bib{Raeburn:Graph_algebras}{book}{
  author={Raeburn, Iain},
  title={Graph algebras},
  series={CBMS Regional Conference Series in Mathematics},
  volume={103},
  publisher={Amer. Math. Soc.},
  place={Providence, RI},
  date={2005},
  pages={vi+113},
  isbn={0-8218-3660-9},
  review={\MR {2135030}},
}

\bib{Renault:Cartan.Subalgebras}{article}{
  author={Renault, Jean},
  title={Cartan subalgebras in $C^*$\nobreakdash -algebras},
  journal={Irish Math. Soc. Bull.},
  number={61},
  date={2008},
  pages={29--63},
  issn={0791-5578},
  review={\MR {2460017}},
  eprint={http://www.maths.tcd.ie/pub/ims/bull61/S6101.pdf},
}



\begin{bibdiv}
  \begin{biblist}
    \bibselect{references}
  \end{biblist}
\end{bibdiv}
\end{document}